\documentclass{article}

\usepackage[english]{babel}
\usepackage[utf8]{inputenc}
\usepackage{amsmath,amsthm,amsfonts,amssymb}
\usepackage{graphicx}
\usepackage{multicol}
\usepackage{scalerel}
\usepackage{enumerate}
\usepackage[dvipsnames]{xcolor}
\usepackage[all]{xy}
\usepackage{tikz-cd}
\usepackage{bm}
\input xypic
\usepackage[colorinlistoftodos]{todonotes}
\definecolor{laura}{rgb}{.6,.2,.85}
\definecolor{carolyn}{rgb}{.5,.8,.5}
\definecolor{constanze}{rgb}{0.5,1.0,0.83}

\newtheorem{theorem}{Theorem}
\newtheorem{lemma}[theorem]{Lemma}
\newtheorem{corollary}[theorem]{Corollary}
\newtheorem{proposition}[theorem]{Proposition}

\theoremstyle{definition}

\theoremstyle{remark}
\newtheorem{remark}[theorem]{Remark}

\newcommand{\Z}{\mathbb{Z}}
\newcommand{\Q}{\mathbb{Q}}
\newcommand{\N}{\mathbb{N}}
\newcommand{\ds}{\displaystyle}
\newcommand{\C}{\mathcal{C}}
\newcommand{\B}{\mathcal{B}}
\newcommand{\D}{\mathcal{D}}
\newcommand{\Sp}{\mathcal{S}}

\newcommand{\E}{E(1)}

\DeclareMathOperator{\HOM}{Hom}
\DeclareMathOperator{\Ho}{Ho}
\DeclareMathOperator{\Ext}{Ext}
\DeclareMathOperator{\Hom}{Hom}
\DeclareMathOperator{\End}{End}
\DeclareMathOperator{\Ima}{Im}

\title{Endomorphisms of Exotic Models}

\author{Eugenia Ellis, Constanze Roitzheim, Laura Scull and Carolyn Yarnall}

\date{\today}

\begin{document}
\maketitle

\begin{abstract}
We calculate the endomorphism dga of Franke's exotic algebraic model for the $K$-local stable homotopy category at odd primes. We unravel its original abstract structure to give explicit generators, differentials and products.
\end{abstract}

\section*{Introduction}

The stable homotopy category $\Ho(\Sp)$ is a large and complex category. 
Thus it becomes natural to break it up.  First we break it into its $p$-local parts $\Ho(\Sp_{(p)})$, and then these are broken into smaller, atomic pieces. These pieces are described by the \emph{chromatic localisations} $\Ho(L_{n}\Sp)$, $n \in \mathbb{N}$. (Note that the prime $p$ is traditionally absent from notation.) We can think of the stable homotopy category as a city with a tower block with infinitely many floors for each prime, the first $n$ floors being described by $\Ho(L_{n}\Sp)$ and the $n^{th}$ floor of each tower block being described by $\Ho(L_{K(n)}\Sp)$ where $K(n)$ is the $n^{th}$ Morava $K$-theory.   

\bigskip

\setlength{\unitlength}{1cm}
\begin{picture}(15, 4)(-2,0)
\put(0,0){ \line(1,0){10} }

\put(1,0){ \line(0,1){3} }
\put(2,0){ \line(0,1){3} }
\put(2.5,0){ \line(0,1){3} }
\put(3.5,0){ \line(0,1){3} }
\put(4,0){ \line(0,1){3} }
\put(5,0){ \line(0,1){3} }

\put(1,3){ \line(1,0){1} }
\put(2.5,3){ \line(1,0){1} }
\put(4,3){ \line(1,0){1} }

\put(1,0.5){ \line(1,0){1} }
\put(2.5,0.5){ \line(1,0){1} }
\put(4,0.5){ \line(1,0){1} }

\put(1,1){ \line(1,0){1} }
\put(2.5,1){ \line(1,0){1} }
\put(4,1){ \line(1,0){1} }

\put(1,1.5){ \line(1,0){1} }
\put(2.5,1.5){ \line(1,0){1} }
\put(4,1.5){ \line(1,0){1} }

\put(1,2.5){ \line(1,0){1} }
\put(2.5,2.5){ \line(1,0){1} }
\put(4,2.5){ \line(1,0){1} }

\put(6, 2){.....}

\put(1.6, 2){.}
\put(1.6, 2.1){.}
\put(1.6, 1.9){.}

\put(3.1, 2){.}
\put(3.1, 2.1){.}
\put(3.1, 1.9){.}

\put(4.6, 2){.}
\put(4.6, 2.1){.}
\put(4.6, 1.9){.}

\put(1.3, 0.2){\tiny$n=0$}
\put(1.3, 0.7){\tiny$n=1$}
\put(1.3, 1.2){\tiny$n=2$}
\put(1.2, 2.7){\tiny$n=\infty$}

\put(5.5,0.8){\vector(1,0){2}}
\put(6.4, 0.4){\bf p}

\put(0.5,0.3){\vector(0,1){2}}
\put(0,1){\bf n}

\put(0, 3.5){\footnotesize\it Visualising $\Ho(\Sp)$ in relation to $\Ho(L_{K(n)}\Sp)$:}
\end{picture}

The ``ground floor'', $\Ho(L_{K(0)}\Sp)$, is given by rational homotopy theory; this is the same for all primes. The first and ground floor, $\Ho(L_{1}\Sp)$, are governed by $p$-local topological $K$-theory, which is related to vector bundles. The next level, $\Ho(L_{2}\Sp)$, is related to elliptic curves,  but is already much more complicated to describe, while the higher levels are valuable for their structural contribution to the bigger picture rather than any individual computational merits. 

Schwede showed in \cite{Schwede07} that the triangulated structure of $\Ho(\Sp)$ determines the entire higher homotopy information of spectra, that is, it determines the underlying model category up to suitable equivalence. In other words, the stable homotopy category is \emph{rigid}. This is particularly interesting because examples of rigidity are usually hard to find. A natural question to follow is whether the atomic building blocks $\Ho(L_{n}\Sp)$ are also rigid.   Franke showed in \cite{Franke96} that for $n=1$ and $p\geq 5$
this is false and $\Ho(L_{n}\Sp)$ are not rigid   by constructing an algebraic counterexample.  Note  that the Franke's result in \cite{Franke96} is formulated for $n^2+n<2p-2$. This version contains a gap which is pointed out in \cite{Patchkoria}, and partially filled in \cite{Patchkoria2}.  

The second author showed in \cite{Roitzheim07} that in contrast, in the case of $n=1$ and $p=2$, the $K$-local stable homotopy category $\Ho(L_1\Sp)$ is rigid. To this day it is rather mysterious why counterexamples exist for $p\geq 5$ but not for $p=2$, and what the situation is like outside of the range covered by Franke and Roitzheim.  For $p=3$ there is an equivalence but it seems from \cite{Patchkoria2} that is unknown whether the equivalence is triangulated. 

Franke's model is \emph{algebraic}, which means that it is model enriched over the model category of chain complexes. Therefore it makes sense to direct the study of exotic models to algebraic models. For example, is Franke's model the only algebraic model for $\Ho(L_1\Sp)$? Or are all exotic models for $\Ho(L_1\Sp)$ algebraic? 

By Morita theory, algebraic model categories which have a single compact generator are determined by an endomorphism dga with homology and Massey products. To get a grip on those uniqueness questions we have to understand the endomorphism dgas:   if there was a unique endomorphism dga, then there would also  be a unique algebraic model. This has partially been answered in \cite{Roitzheim15} but it does not seem feasible to approach this by hand due to the rapidly increasing complexity of the computations.

Thus,  in order to work towards a greater understanding of algebraic models, their uniqueness,  and ultimately the stable homotopy category,  we are going to look at the endomorphism dga of Franke's exotic models. This construction used  many abstract ingredients such as injective resolutions of $E(1)_*E(1)$-comodules, Adams operations, quasi-periodicity and $v_1$-self maps.   The goal of this paper is to carefully unravel these abstractions  in order to arrive at the $\Z_p$-module structure of the dga in question. We hope that going through and turning the  abstract machinery into concrete numbers will contribute to the greater picture by allowing for direct calculations in the future. 

\bigskip
This paper is organised as follows. In Section \ref{sec:algmodels} we  recall some background on endomorphism dgas and the context that we are using them in. In Section \ref{sec:franke} we give a summary of the construction and properties of Franke's exotic model for $\Ho(L_1\Sp)$. In Section \ref{sec:enddga} we perform first steps to simplify the endomorphism dga of a compact generator of Franke's model, showing that some pieces are trivial.  In Section \ref{sec:sequences} we show how the endomorphism dga can be expressed explicitly in terms of sequences with coefficients in $\Z_{p}$, using work of  \cite{Clarke-Crossley-Whitehouse}.  In Sections \ref{sec:n0} and \ref{sec:nnot0} we use the sequence representation to do an explicit calculation of the homology of the endomorphism dga, verifying that it gives the expected result.   We conclude in Section \ref{sec:massey} by verifying that the product and Massey products also give the expected result.

The authors thank the organizers of the WIT II conference, the Banff International Research Station for hosting us, and the AWM for providing travel support.  
The first author is grateful to the Universidad de la Rep\'ublica - CSIC for its support and for travel funds. 
The second author would like to thank the University of Kent Faculty of Sciences Research Fund as well as SMSAS for travel funds, and would furthermore like to thank Andrew Baker, David Barnes and Sarah Whitehouse for interesting discussions.


\section{Algebraic Models}\label{sec:algmodels}

The basic goal is to study the $K$-local stable homotopy category at an odd prime $p$.  We assume that the reader is familiar with basic notions regarding stable model categories and Bousfield localisation, see e.g. \cite{Barnes-Roitzheim:Local Framings}. For background on $K$-theory and related topics, see \cite{Bousfield}. Recall that $K$-theory splits into
\[
K= \bigvee\limits_{i=0}^{p-2} \Sigma^{2i} E(1)
\]
where $E(1)$ is the Adams summand with $E(1)_*=\mathbb{Z}_{(p)}[v_1, v_1^{-1}]$, $|v_1|=2p-2$. Thus, $L_{K_{(p)}}=L_{E(1)}$, which is commonly denoted by $L_1$. 

To study the $K_{(p)}$-local stable homotopy category  $\Ho(L_1\Sp)$, we will study the existence of 
 \emph{algebraic model categories}:    a stable  $Ch(\mathbb{Z})$-model category $\C$ in the sense of \cite[Appendix A]{Dugger},  such that there is an equivalence of triangulated categories
\[
\Phi: \Ho(L_1\Sp) \longrightarrow \Ho(\C).
\]

If $\C$ is an arbitrary stable model category, it can be very hard to understand it, or to compare  $L_1\Sp$ with $\C$. The following result \cite[Theorem 3.1.1]{Schwede-Shipley} gives a more concrete way to approach $\C$.   Recall that an object $X \in \Ho(\C)$ is \emph{compact} if the functor $\Ho(\C)(X,-)$ commutes with arbitrary coproducts. $X$  is a \emph{generator} if the full subcategory of $\Ho(\C)$ containing $X$ which is closed under coproducts and exact triangles is again $\Ho(\C)$ itself. Then we have the following result.  

\begin{theorem}\label{thm:schshi}[Schwede-Shipley]
Let $\C$ be a simplicial proper, stable model category with a compact generator $X$. Then there exists a chain of simplicial Quillen equivalences between $\C$ and module spectra over the endomorphism ring spectrum of $X$,
\[
\C \simeq \mbox{mod-}\End(X).
\]
\end{theorem}

Note that  the assumption that $\C$ is simplicial is not a significant restriction, see e.g. \cite{Dug01}. 

The category $\Ho(L_1\Sp)$ possesses the sphere $L_1S^0$ as a compact generator.  Thus if
$
\Phi: \Ho(L_1\Sp) \longrightarrow \Ho(\C)
$
is a triangulated equivalence as above, we can use (a fibrant and cofibrant replacement of) $X=\Phi(L_1S^0)$ as a compact generator for $\Ho(\C)$. 

From Theorem \ref{thm:schshi}, we know that the endomorphism ring spectrum $\End(X)$ satisfies
\[
\pi_*(\End(X))\cong\Ho(\C)(X,X)_{*}.
\]
Combining this with our triangulated equivalence,  we have
\[
\pi_*(\End(X))\cong\Ho(\C)(X,X)_{*}\cong \pi_*(L_1S^0).
\]

Now if we additionally assume that $\C$ is an algebraic category,   \cite[Proposition 6.3]{Dugger-Shipley} gives us the following about the endormorphism spectrum:  
\begin{theorem}
Let $\C$ be an algebraic model category with a fibrant and cofibrant compact generator $X$. Then the endomorphism ring spectrum $\End(X)$ is weakly equivalent to the generalised Eilenberg-Mac Lane spectrum of the endomorphism dga $\C(X,X).$
\end{theorem}

Moreover, for $X \cong \Phi(L_1S^0)$,
the endomorphism dga $\C(X,X)$ satisfies (\cite[Lemma 2.1]{Roitzheim15}):

\begin{itemize}
\item $H^*(\C(X,X)) = \Ho(\C)(X,X)_{*}=\pi_*(L_1S^0)$.
\item Under the above, the Massey products of $\C(X,X)$ coincide with the Toda brackets of $\pi_*(L_1S^0)$.
\end{itemize}

Thus we see that in order to understand algebraic models $\C$ for $L_1\Sp$ it is vital to understand the endomorphism dga of a compact generator.   In the next section, we will describe a specific algebraic model $\C$ that will be the focus of this paper, and also take a closer look at its compact generator.

\section{Franke's model and its compact generator}\label{sec:franke}

In this section we are going to give a brief description of the particular algebraic model  for $\Ho(L_1\Sp)$ that we will be looking at in detail in the subsequent sections.  This was developed by Franke  \cite{Franke96};  further details are available in  \cite{Roitzheim08} (and  \cite{Patchkoria2} for the triangulated structure).  In what follows, we will use notation consistent with \cite{Roitzheim08}.

To begin, we consider the 
 category $\B$,  an abelian category which is equivalent to $\E_*\E$-comodules that are concentrated in degrees $0\,\, \mbox{mod}\,\, 2p-2$. (Note that in  \cite{Bousfield}, Bousfield denotes this category by $\B(p)_*$.) We can think of $\E_*\E$-comodules as modules over $\E_*$ with an action of the Adams operations. Furthermore, the category $\B$ is equipped with self-equivalences
\[
T^{j(p-1)}:\B \longrightarrow \B \,\,\, (j \in \mathbb{Z})
\]
each of which is the identity on the underlying $\E_*$-modules but changes the Adams operation $\Psi^k$ by a factor of $k^{j(p-1)}$.

\bigskip
Now we consider \emph{twisted chain complexes}  $\C^{2p-2}(\B)$ on $\B$.   An object of $\C^{2p-2}(\B)$ is a cochain complex $C^*$ with $C^i \in \B$ together with an isomorphism
\[
\alpha_C: T^{(p-1)}(C^*) \longrightarrow C^*[2p-2]=C^{*+2p-2}.
\]
Morphisms in this category are cochain maps $f: C^* \longrightarrow D^*$ which are compatible with those isomorphisms, i.e. for which there is a commutative diagram
\[
\xymatrix{
T^{(p-1)}(C^\ast)\ar[rr]^{\alpha_C}\ar[d]_{T^{(p-1)}(f)} && C^\ast[2p-2]\ar[d]^{f[2p-2]} \\
T^{(p-1)}(D^\ast)\ar[rr]^{\alpha_D} & &D^\ast[2p-2].
}
\]
We can define a model structure on  $\C^{2p-2}(\B)$  as follows.  
\begin{proposition}[Franke]\label{prop:modelstructure} There is a model structure on  $\C^{2p-2}(\B)$ such that
\begin{itemize}
\item weak equivalences are the quasi-isomorphisms
\item cofibrations are the monomorphisms
\item fibrations are the degreewise split epimorphisms with strictly injective kernel.
\end{itemize}
\end{proposition}
Here, an object $C^*$ is said to be strictly injective if it is levelwise injective and
for each acyclic complex $D^\ast$, the mapping chain complex $\Hom_{\C^{2p-2}(\B)}(D^\ast,C^\ast)^*$ is again acyclic.

Note that the above model structure is a variant of the standard injective model structure on chain complexes. There is no projective-type model structure on $\C^{2p-2}(\B)$,  as $\B$ has enough injectives but not enough projectives. 

Now let  $\D^{2p-2}(\B)$ be the homotopy category of a model category  of  $\C^{2p-2}(\B)$.   This is the exotic algebraic model we are interested in:  

\begin{theorem}[Franke]\label{T:Franke-model}
For $p\geq 5$ there is an equivalence of triangulated categories
\[
\mathcal{R}: \D^{2p-2}(\B) \longrightarrow \Ho(L_1\Sp)
\]
which satisfies
\[
\bigoplus\limits_{i=0}^{2p-3}H^i(C)[-i]\cong E(1)_*(\mathcal{R}(C)).
\]
\end{theorem}

\bigskip
Concerning the equivalence $\mathcal{R}: \D^{2p-2}(\B) \longrightarrow \Ho(L_1\Sp)$,  the notation $\mathcal{R}$ stands for \emph{reconstruction functor}.   Usually one would expect an equivalence between two categories such as the above to have the category of topological origin as its source and the algebraic category as its target. But in this unusual case, the equivalence \emph{reconstructs} a topological object from an algebraic one.  

This reconstruction can be described as follows.  To build a spectrum $X$ from a chain complex $C^*$, one first considers the boundaries $B^i$ of $C^* (1 \leq i \leq 2p-2)$ and the quotients $G^i$ of $C^*$ by its boundaries. Then, one assigns spectra $X_{\beta_i}$ and $X_{\gamma_i}$ to the $B^i$ and $G^i$ respectively, so that
\[
G^i(X)=E(1)_*(X_{\gamma_i})[-i]\,\,\,\mbox{and}\,\,\,B^i(X)=E(1)_*(X_{\beta_i})[-i].
\]

These spectra are now arranged in a crown-shaped diagram 
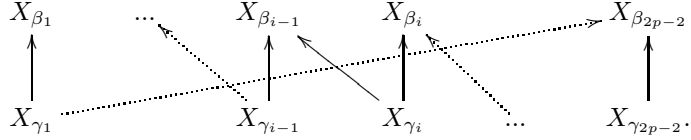
\[
\xymatrix{      X_{\beta_1} & ... & X_{\beta_{i-1}} & X_{\beta_i} & & X_{\beta_{2p-2}} \\
             X_{\gamma_1} \ar[u]\ar@{.>}[urrrrr] & & X_{\gamma_{i-1}}\ar[u]\ar@{.>}[ul] & X_{\gamma_i}\ar[u]\ar[ul] & ... \ar@{.>}[ul]
            & X_{\gamma_{2p-2}}. \ar[u]
}\]

Then the reconstruction spectrum $X=\mathcal{R}(C^*)$ is defined to be  the homotopy colimit of the above diagram. Proving that this defines an equivalence of categories as stated in Theorem \ref{T:Franke-model} is a lengthy progress involving various Adams spectral sequences and diagram chases.   Once it is completed, however, it is not too hard to read off the following:

\begin{lemma}\label{lem:enddga}
The cochain complex $A^*:=\mathcal{R}^{-1}(L_1S^0)$ is $A^i=T^{k(p-1)}(E(1)_*)$ in degrees $i=k(2p-2), k \in \mathbb{Z}$ and $0$ in all other degrees.
\end{lemma}
\qed

\section{The endomorphism dga}\label{sec:enddga}

Recall from Section \ref{sec:algmodels} that in order to understand an algebraic model, we want to study the endomorphism dga of a compact generator.  We know that the 
cochain complex of Lemma \ref{lem:enddga}
\begin{align*}
A^* & =&  \cdots &\longrightarrow & 0 & \longrightarrow & T^{-(p-1)}E(1)_* & \longrightarrow & 0 & \longrightarrow  & \cdots  \\ 
& &  \cdots & \longrightarrow &  0 & \longrightarrow &  E(1)_* & \longrightarrow &  0 & \longrightarrow & \cdots \\& & \cdots & \longrightarrow & 0 & \longrightarrow &  T^{(p-1)}E(1)_* & \longrightarrow & 0 & \longrightarrow & \cdots   
\end{align*}
is a compact generator for $\D^{2p-2}(\B)$. Hence, to understand Franke's model we need to study the endomorphism dga $C^*$ of $A^*$, i.e.
\[
C^*:= \HOM_{\C^{2p-2}(\B)}(A^*, A^*).
\]
 By 
 construction,
\[
H^{t-s}(C^*)= \Ext_\B^{s, t}(\E_*, \E_*)
\]
which  is the $E^2$-term of the $\E_*$-based Adams spectral sequence for $\pi_*(L_1S^0)$.    Examining the degrees shows that this spectral  collapses, giving an isomorphism    $H^n(C^*) = \pi_n(L_1S^0)$.    \bigskip

We intend to  unravel what $C^*$ looks like as $\mathbb{Z}_{(p)}$-module and obtain a concrete description of this chain complex.    
We begin by considering the general form of  any mapping chain complex $\HOM_{\C^{2p-2}(\B)}(X^*, Y^*)$ for arbitrary $X^*, Y^* \in \C^{2p-2}(\B).$   This satisfies
\[
\HOM_{\C^{2p-2}(\B)}(X^*, Y^*)= \D^{2p-2}(\B)(X^*, Y^*).
\]
When $X^*$ and $Y^*$ are  concentrated in one degree up to periodicity, i.e. $$X^*= \prod\limits_{k \in \mathbb{Z}}T^{k(p-1)}X[-k(2p-2)] \,\,\,\mbox{and}\,\,\, Y^*= \prod\limits_{k \in \mathbb{Z}}T^{k(p-1)}Y[-k(2p-2)]$$ for some $X,Y \in \B$, we have
\[
H^{n-i}(\HOM_{\C^{2p-2}(\B)}(X^*, Y^*))=\prod\limits_{i}\Ext_\B^{i,n}(X,Y).
\]
We examine what such a morphism in  $\C^{2p-2}(\B)$ looks like when  $X^*$ is cofibrant and $Y^*$ is fibrant. We will see that all morphisms $f^*: X^* \longrightarrow Y^*+s$ are not only determined by the first $f^0, ..., f^{2p-3} \in \B$ but also solely by the low-degree terms of $X^*$ and $Y^*$. 

\bigskip
Firstly, by definition of the category $\C^{2p-2}(\B)$ in Section \ref{sec:franke}, a morphism satisfies  $$f^{*+2p-2} \cong T^{p-1}(f^*).$$   
This means that for example
a map $f^*: X^* \longrightarrow Y^*$ of degree $0$ is defined by morphisms  $f^i:  X^i \to Y^i$ in $\B$ for $0 \leq i \leq 2p-3$, 
and similarly,  a map $f^*$ of degree $n \in \mathbb{Z}$ is determined by   $f^i: X^i \longrightarrow Y^{i+n}$ for $0 \le i \le 2p-3$. 

\bigskip 
However, we also claim that a morphism of degree $n=(2p-2)r+s$, $0 \leq s \leq 2p-3$ is in fact already defined by a morphism of degree $s$ in $\B$ between the lower degrees of $X^*$ and $Y^*$, i.e. the low-degree morphisms define the entire mapping chain complex.

To see this, consider a morphism of degree $2p-2$, determined by 
\begin{eqnarray}
f^0: & X^0 \longrightarrow Y^{2p-2}\cong T^{p-1}(Y^0) \nonumber \\
f^1: & X^1 \longrightarrow Y^{2p-1}\cong T^{p-1}(Y^1) \nonumber \\
\vdots & \vdots \nonumber \\
f^{2p-3}: &  X^{2p-3} \longrightarrow Y^{4p-5} \cong T^{p-1}(Y^{2p-3}). \nonumber
\end{eqnarray}
Consider the map $f^0$ in $\B$.   Recall that  objects in $\B$ are themselves graded, and denote this internal degree by a subscript. Therefore, by definition of $T^{p-1}$,
\[
f^0=f^0_*: X^0_* \longrightarrow Y^{2p-2}_* \cong T^{p-1}(Y^0)_* \cong Y^0_{2p-2}.
\]
Any morphism in $\B$ $$F_*: M_* \longrightarrow N_*$$  is given by a $\mathbb{Z}_{(p)}$-module map satisfying $F \circ \Psi^k = \Psi^k \circ F$ for the Adams operation $\Psi^k$, {} $k \in \mathbb{Z}_{(p)}$. Thus, a morphism $$F: M_* \longrightarrow T^{p-1}(N)_*=N_{*+2p-2}$$ is a $\mathbb{Z}_{(p)}$-module map $$F: M_* \longrightarrow N_{*+2p-2}$$ satisfying $F(\Psi^k x)= k^p F(X).$ Now we also have \cite[Section 4.2]{Bousfield},
\[
\Psi^k(v_1 \cdot y)=k^{p-1}v_1 \cdot \Psi^k(y) = k^p v_1 y,
\]
 which means that $F$ factors as
\[
M_* \stackrel{G}{\longrightarrow} N_* \stackrel{\cdot v_1}{\longrightarrow} N_{*+2p-2}
\]
where $G$ is a map in $\B$ of degree $0$, and multiplication by $v_1$ is an isomorphism. \bigskip

Returning to our map of chain complexes $$f^0: X^0 \longrightarrow Y^{2p-2},$$ we see that $f^0$ can be factored as $f^0=v_1 \cdot g_0$, where $g_0: X^0 \longrightarrow Y^0$ is a morphism in $\B$. Similarly,  any map $$f^*: X^* \longrightarrow Y^* \in \C^{2p-2}(\B)$$ of degree $n=(2p-2)r + s$, $0 \leq s \leq 2p-3$ is determined by a map of degree $s$. Thus,
\[
\HOM_{\C^{2p-2}(\B)}(X^*, Y^*)^*= \prod\limits_{n \in \mathbb{Z} } \Hom_\B(X^*, Y^{*+n})=\prod\limits_{0 \leq i, s \leq 2p-3} \Hom_\B(X^i, Y^{i+s}).
\]
Note that we have not yet considered the internal grading. The object $$\Hom_\B(X^i, Y^{i+s})$$ is a graded $\E_*$-module, with the grading coming from the internal grading in $\B$ on $X^i=X^i_*$ and $Y^{i+s}=Y^{i+s}_*$. We say that an element in $ \Hom_\B(X^i, Y^{i+s})$ has degree $t$ if it raises the internal degree by $t$. As we will consider each degree separately, we use  $ \Hom_\B(X^i_*, Y^{i+s}_{*+t})$ to denote those morphisms in $\B$ that raise degree by $t$.  So in our notation,  this is only a $\mathbb{Z}_{(p)}$-module and \emph{not} an $\E_*$-module.   In particular,  $\Hom_\B(X^i, Y^{i+s})$ is \emph{not} a graded object. 

\bigskip
Taking this internal degree into account, we define $\HOM_{\C^{2p-2}(\B)}(X^*, Y^*)^*$ to be the chain complex defined in degree $n$ by
\[
\HOM_{\C^{2p-2}(\B)}(X^*, Y^*)^n=\prod\limits_{\substack{n=t-s, \\ 0\leq i,  s \leq 2p-3} } \Hom_\B(X^i_{*-t}, Y^{i+s}_*),
\]
i.e. as shown earlier, the mapping chain complex is defined only by low-degree terms of the chain complexes as well as low-degree morphisms.
(Recall that we are assuming that  $X^*$ is cofibrant and $Y^*$ is fibrant.  If this is not the case,  cofibrant  and  fibrant replacements need to be applied.)   The $n^{th}$ differential is given by
\[
d_A \circ f + (-1)^{n+1} f \circ d_B.
\]
The grading is consistent with the equivalence given in Lemma \ref{lem:enddga}
\[
\pi_{t-s}(L_1S^0)=\Ext^{s,t}_\B(\E_*, \E_*) = H^{t-s} \HOM(A^*, A^*)
\]
where $A^*$ is the compact generator.  Explicitly $A^*$ is  the cochain complex which is $A^i=T^{k(p-1)}(E(1)_*)$ in degrees $i=k(2p-2), k \in \mathbb{Z}$ and $0$ in all other degrees.

\bigskip
In order to apply the above discussion to our endomorphism complex, we need to find a fibrant and cofibrant replacement for $A^*$. The model structure of  Proposition \ref{prop:modelstructure} implies that any object in $\C^{2p-2}(\B)$ is cofibrant, so in fact we only need a fibrant replacement.

 To produce a fibrant replacement, we will use an injective resolution
\begin{equation}\label{inj-res}
0 \longrightarrow \E_* \longrightarrow I^0 \longrightarrow I^1 \longrightarrow I^2 \longrightarrow 0
\end{equation}
of $E(1)_*$ as an $\E_*\E$-comodule.  Since $A^*$ is $\E_*$ repeated periodically using the self-equivalence $T^{(p-1)}$,  we will obtain an injective resolution of $A^*$  by taking the injective resolution above and repeating it periodically, again applying the self-equivalence $T^{(p-1)}$.   Since $p$ is odd and the injective dimension of $\B$ is 2 (as is the injective dimension of $\E_*\E$-comod)  \cite[Section 7]{Bousfield}, the pieces from the injective resolution do not overlap in the cochain complex.  

For the injective resolution in (\ref{inj-res}), we will use the standard injective resolution by Adams-Baird-Ravenel \cite{ABR}
\begin{equation}\label{eqn:resolution}
0 \longrightarrow \E_* \longrightarrow \E_*\E \xrightarrow{(\Psi^r-1)_*} \E_*\E \xrightarrow{\,\,q\,\,} \E_* \otimes \mathbb{Q} \longrightarrow 0
\end{equation}
where  $r$ is a unit of the cyclic group $(\mathbb{Z}/p^2)^\times$, $\Psi^r$ is the $r^{th}$ Adams operation and $q$ is induced by the map $E(1) \longrightarrow H\mathbb{Q}$ that is a rational homotopy isomorphism in degree $0$ and trivial otherwise.

Note that this resolution $I$ does not consist of injective comodules but of \emph{relative injective} comodules, see \cite[Definition 3.1.1]{Hovey:comodules}, i.e. the functor $\Hom_\B(-,I)$ sends split short exact sequences of $\E_*$-modules to short exact sequences. By definition,
\[
\Ext^*_\B(\E_*, \E_*) = H^*(\Hom_\B(\E_*, J))
\]
where $J$ is an injective resolution of $\E_*$. As $J$ is injective, one also has
\[
\Hom_\B(\E_*,J) \simeq \Hom_\B(J, J).
\]
By \cite[Lemma 3.1.4]{Hovey:comodules} the above is quasi-isomorphic to $\Hom_\B(\E_*, I)$ with $I$ our relative injective resolution. Now consider the split exact sequence of $\E_*$-modules
\[
0 \longrightarrow \E_* \longrightarrow I \longrightarrow K \longrightarrow 0
\]
where $K$ is the cokernel of the first map. This $K$ is bounded above and below as well as acyclic, so it is contractible. Thus, by \cite[Lemma 3.3.3]{Hovey:comodules}, every map $K \longrightarrow I$ is chain homotopic to the zero map, so $\Hom_\B(K,I)\simeq 0$ and consequently
\[
\Hom_\B(\E_*, I) \simeq \Hom_\B(I, I) \simeq \Hom_\B(J, J),
\]
which is what we are using. 

\bigskip
Thus we create a relative injective replacement equivalent to the fibrant replacement 
\begin{align*}
(A^{fib})^* & =&  \nonumber\\
 \cdots 0 & \longrightarrow & T^{-(p-1)}I^0& \longrightarrow &  T^{-(p-1)}I^1 & \longrightarrow & T^{-(p-1)}I^2 & \longrightarrow & 0 &\longrightarrow & \cdots \nonumber \\
\cdots 0 & \longrightarrow & I^0 & \longrightarrow & I^1 & \longrightarrow &  I^2 & \longrightarrow & 0 & \longrightarrow & \cdots \nonumber \\
\cdots 0 & \longrightarrow &  T^{(p-1)}I^0& \longrightarrow &  T^{(p-1)}I^1 & \longrightarrow & T^{(p-1)}I^2&   \longrightarrow &  0 & \longrightarrow & \cdots \nonumber
\end{align*}

In other words, $$(A^{fib})=\mathbb{R}I=\prod\limits_{k \in \mathbb{Z}} T^{k(p-1)}I[-k(2p-2)]$$
with $$I=(...0 \rightarrow I^0 \rightarrow I^1 \rightarrow I^2 \rightarrow 0 ...) \in Ch(\B)$$ and $[n]$ denoting the $n^{th}$ suspension.

Returning to the definition of the endomorphism complex $C^*$, we have that  $$C^*:=\Hom_{\C^{2p-2}(\B)}((A^{fib})^*, (A^{fib})^*)$$ is entirely determined by the terms of the form
\[
\Hom_\B(I^j, I^k),  \textup{ \,\, where \,\,} i,j \in \{0,1,2\}.
\]
So we have to calculate nine potential terms:  
\begin{multline}
C^n := \Hom_{\C^{2p-2}(\B)}((A^{fib})^*, (A^{fib})^*)^n = \prod\limits_{n=t-s, i}\Hom_\B((I^i)_{*-t}, (I^{i+s})_*) \nonumber \\
\end{multline}
\begin{align*}
& =&  &\Hom_\B((I^0)_{*-n}, (I^{0})_*) & \times &  \Hom_\B((I^0)_{*-(n-1)}, (I^{1})_*) & \times&  \Hom_\B((I^0)_{*-(n-2)}, (I^{2})_*) \nonumber \\
& &\times &  \Hom_\B((I^1)_{*-(n+1)}, (I^{0})_*) & \times&  \Hom_\B((I^1)_{*-n}, (I^{1})_*) & \times&  \Hom_\B((I^1)_{*-(n-1)}, (I^{2})_*) \nonumber \\
& & \times & \Hom_\B((I^2)_{*-(n+2)}, (I^{0})_*) & \times&  \Hom_\B((I^2)_{*-(n+1)}, (I^{1})_*) & \times & \Hom_\B((I^2)_{*-n}, (I^{2})_*) \nonumber
\end{align*}
and specify the differentials between those terms.  

\bigskip
Since the terms appearing in the sequence (\ref{eqn:resolution}) are either of the form $E(1)_*E(1)$ or $E(1)_* \otimes \mathbb{Q}$, the nine terms above can be grouped into four types of the following form:
\begin{enumerate}[(I)]
\item $\Hom_\B(\E_{*-t}\E, \E_*\E)$
\item $\Hom_\B(\E_{*-t}\E, \E_* \otimes \mathbb{Q})$
\item $\Hom_\B(\E_{*-t}\otimes \mathbb{Q}, \E_*\E)$
\item $\Hom_\B(\E_{*-t}\otimes \mathbb{Q}, \E_*\otimes\mathbb{Q})$
\end{enumerate}
All of the above are trivial unless $t$ is a multiple of $2p-2$. By \cite[Appendix A1]{Ravenel Green Book} we have the following natural isomorphism 
\begin{equation}\label{E:rav}
\Hom_{\E_*}(M,N) \cong \Hom_\B(M, \E_*\E \otimes_{\E_*} N)
\end{equation}
for $\E_*$-modules $M$ and $N$. Applying this to the terms above yields the following. 

\bigskip
{\bf Type (I)} 
The isomorphism (\ref{E:rav}) gives
\begin{eqnarray*}\Hom_\B(\E_{*-t}\E, \E_*\E)& \cong&  \Hom_{\E_*}(\E_{*-t}\E, \E_*) \\  & \cong& \Hom_{\mathbb{Z}_{(p)}}(\E_0\E, \mathbb{Z}_{(p)}(v_1^k)) \textup{ \phantom{W} for } t=(2p-2)k\end{eqnarray*} 
\bigskip

{\bf Type (II)}
Here, we have to distinguish between $t=0$ and $t \neq 0$. Let us begin with $t=0$. By \cite{Adams-Clarke}, $\E_*\E$ consists of Laurent polynomials living in $\mathbb{Q}[u,u^{-1},w,w^{-1}]$ with $|u|=|w|=2p-2$ satisfying certain conditions. (We can think of $u$ and $w$ as ``two copies of $v_1$'' in terms of the $\E_*$-action.) Furthermore,
\[
\E_*\E \otimes \mathbb{Q}\cong \mathbb{Q}[u, u^{-1}, w, w^{-1}].
\]
Now let $f \in \Hom_\B(\E_*E, E_*\otimes \mathbb{Q}).$ By definition, $f$ is an $\E_*$-module homomorphism, and also the following diagram has to commute.

\[
\xymatrix{ \E_*\E \ar[r]^{f}\ar[d]_{\Delta} & \E_*\otimes\mathbb{Q}=\mathbb{Q}[v_1,v_1^{-1}] \ar[d]^{\psi}  \\
\E_*\E \otimes_{\E_*} \E_*\E \ar[r]^<<<<{1 \otimes f} & \E_*\E \otimes_{\E_*} \E_* \otimes \mathbb{Q} = \E_*\E \otimes \mathbb{Q}. \\
} \]

For the coactions, we have 
\[
\Delta{u^iw^j}=u^iw^{2j} \,\,\,\mbox{and}\,\,\,\psi(v_1^i)=u^i,
\]
i.e. $w$ is group-like, see e.g. \cite{AHS} or \cite{Johnson}.

Since $f$ is an $E_*$-module homomorphism,  we have
\[
f(u^iw^j)=v_1^{i+j}f(1).
\]
Putting this information together we see that the diagram can only commute if $j=0$ or $f$ is the zero map. So, $f$ has to be zero on $w$, i.e. is only supported on Laurent polynomials in $u$. Thus, $f$ can be considered as an $\E_*$-module homomorphism from $\E_* \cong \mathbb{Z}_{(p)}[u, u^{-1}]$ to $\E_* \otimes \mathbb{Q}$. Thus, we have
\[
\Hom_\B(\E_*\E, \E_*\otimes \mathbb{Q})\cong \Hom_{\E_*}(\E_*, \E_* \otimes \mathbb{Q}) \cong \mathbb{Q}.
\]
Note that the last isomorphism holds as we are only considering degree-preserving morphisms. Furthermore, note that the 1 in the last $\mathbb{Q}$ corresponds to the morphism that sends $w$ to 0 and $u$ to 1, which is exactly $q$ from the Adams-Baird-Ravenel resolution (\ref{inj-res}).

\bigskip
Let us now look at the case $t \neq 0$:

Analogously to the previous argument, 
\[
\Hom_\B(\E_*E, \E_{*+t}\otimes \mathbb{Q})\cong \Hom_{\E_*}(\E_*, \E_{*+t} \otimes \mathbb{Q}) \cong \mathbb{Q}.
\]
As a module over $\mathbb{Q}$, this is generated by the function sending $w$ to $0$ and $u$ to $v_1^{s-1}$ for $t=s(2p-2)$. 


\bigskip
{\bf Type (III)}
Every $\E_*$-module is in particular a $\mathbb{Z}_{(p)}$-module, and so every element of (III) is in particular a $\mathbb{Z}_{(p)}$-module homomorphism from $\mathbb{Q}$ to $\mathbb{Z}_{(p)}$. Thus, all terms of the form (III) are zero. 

\bigskip
{\bf Type (IV)}
A morphism
\[
f \in \Hom_\B(\E_* \otimes \mathbb{Q}, \E_*\otimes \mathbb{Q})
\]
is entirely determined by $f(1) \in \E_0 \otimes \mathbb{Q}=\mathbb{Q}$, so for the degree $t=0$ we have
\[
\Hom_\B(\E_* \otimes \mathbb{Q}, \E_*\otimes \mathbb{Q}) = \mathbb{Q}.
\]

For $t \neq 0$, a term of Type (IV) is trivial.
If $t$ is not a multiple of $(2p-2)$, this is already clear for degree reasons. For $t=s(2p-2), s \neq 0$ we have the following. An morphism $$f: \E_* \otimes \mathbb{Q} \longrightarrow \E_{*+t}\otimes \mathbb{Q}$$ in $\mathbb{B}$ is an $\E_*$-module homomorphism which is compatible with Adams operations. The $\E_*$-module homomorphisms are given by
\[
\Hom_{\E_*}(\E_* \otimes \mathbb{Q}, \E_{*+t}\otimes \mathbb{Q}) = \mathbb{Q},
\]
which is generated over $\mathbb{Q}$ by the map that sends $1$ to $v_1^s$. The Adams operations are given by $\Psi^k(v_1^i)=k^{i(p-1)}v_1^s$, in particular $\Psi^k(v_1^2)=k^{2(p-1)}v_1^2$. 

We also know that $\E_{*+t}$ is also isomorphic in $\B$ to $T^{s(p-1)}\E_*$, therefore 
\[
\Hom_\B(\E_* \otimes \mathbb{Q}, \E_{*+t}\otimes \mathbb{Q}) \cong \Hom_\B(\E_* \otimes \mathbb{Q}, T^{s(p-1)}\E_*).
\]
There, by definition of $T$,
\[
\Psi^k(v_1^2)=k^{s(p-1)}\Psi^k_{old}(v_1^2)=k^{s(p-1)}(k^{(2p-2)}v_1^2)=k^{(s+2)(p-1)}v_1^2.
\]
However, this can only be equal to the previously calculated $k^{2(p-1)}v_1^2$ if $s=0$, which proves that for $t \neq 0$,
\[
\Hom_\B(\E_* \otimes \mathbb{Q}, \E_{*+t}\otimes \mathbb{Q}) = 0.
\]

\bigskip
Now that we have identified the forms of the terms I--IV in the endomorphism complex, let us consider the  differentials.  A differential from $C^n$ to $C^{n+1}$ is of the form $d \circ f + (-1)^{n+1} f\circ d$. We illustrate its individual parts in the diagram below, where a solid arrow represents a possible nontrivial $d \circ f$ and a dashed arrow represents a possible nontrivial $f \circ d$.  In addition, each term has been labeled with its type (I-IV).  

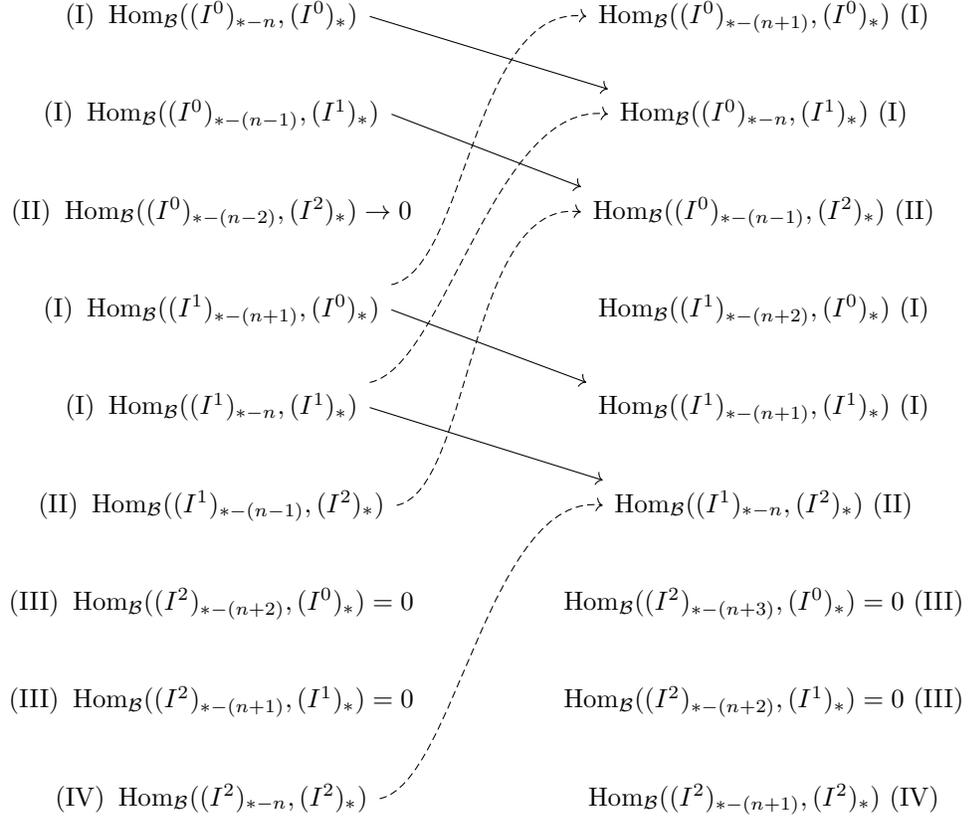
\begin{tikzcd}
\mbox{(I) }\Hom_\B((I^0)_{*-n}, (I^{0})_*)  \arrow[rrd, start anchor=east, end anchor=north west] && \Hom_\B((I^0)_{*-(n+1)}, (I^{0})_*) \mbox{ (I)}\\
\mbox{(I) }\Hom_\B((I^0)_{*-(n-1)}, (I^{1})_*) \arrow[rrd, start anchor=east, end anchor=north west] && \Hom_\B((I^0)_{*-n}, (I^{1})_*)\mbox{ (I)}\\
\mbox{(II) }\Hom_\B((I^0)_{*-(n-2)}, (I^{2})_*) \rightarrow 0 && \Hom_\B((I^0)_{*-(n-1)}, (I^{2})_*)\mbox{ (II)}\\
\mbox{(I) }\Hom_\B((I^1)_{*-(n+1)}, (I^{0})_*) \arrow[rrd, start anchor=east, end anchor=north west] \arrow[rruuu, dashed, start anchor=north east, end anchor=west,controls={+(1,0.2) and +(-1.5,0.2)}]&& \Hom_\B((I^1)_{*-(n+2)}, (I^{0})_*)\mbox{ (I)}\\
\mbox{(I) }\Hom_\B((I^1)_{*-n}, (I^{1})_*) \arrow[rrd, start anchor=east, end anchor=north west] \arrow[rruuu, dashed, start anchor=north east, end anchor=west,controls={+(1,0.2) and +(-1.5,0.2)}] && \Hom_\B((I^1)_{*-(n+1)}, (I^{1})_*)\mbox{ (I)}\\ 
\mbox{(II) }\Hom_\B((I^1)_{*-(n-1)}, (I^{2})_*) \arrow[rruuu, dashed, start anchor=east, end anchor=west,controls={+(1,0.2) and +(-1.5,0.2)}] && \Hom_\B((I^1)_{*-n}, (I^{2})_*)\mbox{ (II)}\\
\mbox{(III) }\Hom_\B((I^2)_{*-(n+2)}, (I^{0})_*)=0 && \Hom_\B((I^2)_{*-(n+3)}, (I^{0})_*)=0\mbox{ (III)}\\
\mbox{(III) }\Hom_\B((I^2)_{*-(n+1)}, (I^{1})_*)=0 &&\Hom_\B((I^2)_{*-(n+2)}, (I^{1})_*)=0\mbox{ (III)}\\ 
\mbox{(IV) }\Hom_\B((I^2)_{*-n}, (I^{2})_*) \arrow[rruuu, dashed, start anchor=east, end anchor=west, controls={+(1,0.2) and +(-1.5,0.2)}] && \Hom_\B((I^2)_{*-(n+1)}, (I^{2})_*)\mbox{ (IV)}
\end{tikzcd}
\bigskip

For any other values of $n$, the dga will be zero.  Combining this information with the interpretations of Terms (I)-(IV), we see that the non-zero terms of the endomorphism dga look like:  
 
\hspace{-1cm}\begin{tikzcd}[column sep=-15ex, row sep=10ex]
 C^{(2p-2)k-1} \ar[r] & C^{(2p-2)k} \ar[r] & C^{(2p-2)k+1} \ar[r] & C^{(2p-2)k+2} \\
  & \Hom_\B(\E_{*-n}\E, \E_*\E) \ar[rd,"\Psi_*"] &   & \\
\Hom_\B(\E_{*-n}\E, \E_*\E) \ar[ru,,"\Psi^*"] \ar[rd,"\Psi_*"] &   & \Hom_\B(\E_{*-n}\E, \E_*\E) \ar[rd,"q_*"] & \\
  & \Hom_\B(\E_{*-n}\E, \E_*\E) \ar[ru,,"\Psi^*"] \ar[rd,"q_*"] &   & \Hom_\B(\E_{*-n}\E, \E_*\otimes\Q) \\
  &   & \Hom_\B(\E_{*-n}\E, \E_*\otimes\Q) \ar[ru,,"\Psi^*"]\\ 
  &  \Hom_\B(\E_{*-n}\otimes\Q, \E_*\otimes\Q)\ar[ru,,"q^*"]
 \end{tikzcd}

\bigskip
Here, $\Psi^*$, $\Psi_*$, $q^*$ and $q_*$ refer to (pre)composing with $\Psi=(\Psi^r-1)$ and $q$ from the Adams-Ravenel-Baird resolution (\ref{eqn:resolution}). 


\begin{remark}
When $p=3$, degree reasons do not rule out a differential $$C^{(2p-2)k+2} \rightarrow C^{(2p-2)k+3}.$$ However, the actual definition of the differential in terms of $\Psi$ and $q$ means that no nontrivial such differential exists.

\end{remark}

\section{Reinterpretation as Sequences}\label{sec:sequences}

We now turn to creating an explicit description of the sequence described in the previous section.  As noted above, terms of Type (III) are trivial, terms of Type (II) give a single copy of $\mathbb{Q}$.  We here consider the other, not so simple terms of Type (I).   

As mentioned above, by \cite[Appendix A1]{Ravenel Green Book},
\begin{equation} \label{E:rav2}
\Hom_\B(\E_{*-n}\E, \E_*\E) \cong \Hom_{\E_*}(\E_*\E, \E_{*+n}).
\end{equation}
Since $\E_*\E$ is free as an $\E_*$-module \cite[Theorem 2.1]{Adams-Clarke}, 
\[
\Hom_{\E_*}(\E_*\E, \E_{*+n}) \simeq
\Hom_{\mathbb{Z}_{(p)}}(\E_0\E, \mathbb{Z}_{(p)}[v_1^k]), \,\,\, n=(2p-2)k.
\]
This dual has been considered in \cite{Clarke-Crossley-Whitehouse}, where it is shown that 
\[
\Hom_{\mathbb{Z}_{(p)}}(\E_0\E, \mathbb{Z}_{(p)}[v_1^k]) \cong \E^0\E.
\]

Furthermore, by \cite[Theorem 6.2]{Clarke-Crossley-Whitehouse} this can be uniquely expressed as a formal series 
\[
\E^0\E \cong \{ \sum\limits_{n\ge 0} a_m \Theta_m(\Psi^r) \,\,|\,\, a_m \in \mathbb{Z}_{(p)}\}
\]
see also \cite[Proposition 18]{Strong-Whitehouse}. Here,   $\Theta_m$ is an explicit polynomial in the Adams operation $\Psi^r$   (where $r$ a generator of $(\mathbb{Z}/p^2)^\times$)  defined as follows:    \cite[Definition 6.1]{Clarke-Crossley-Whitehouse}: 
\begin{eqnarray}
\Theta_0(\Psi^r) &= & 1, \nonumber\\ \Theta_1(\Psi^r) & = & (\Psi^r-1),\nonumber\\ \Theta_2(\Psi^r) &=& (\Psi^r-1)(\Psi^r-r),\nonumber\\ \Theta_3(\Psi^r)& = & (\Psi^r-1)(\Psi^r-r)(\Psi^r-r^{-1}), \nonumber \\ \Theta_4(\Psi^r)&=&(\Psi^r-1)(\Psi^r-r)(\Psi^r-r^{-1})(\Psi^r-r^2), \nonumber\\ \Theta_5(\Psi^r)&=&(\Psi^r-1)(\Psi^r-r)(\Psi^r-r^{-1})(\Psi^r-r^2)(\Psi^r-r^{-2}) \nonumber \\ \mbox{etc.} & & \nonumber
\end{eqnarray}

This means that we can view the elements of $\Hom_\B(\E_{*-n}\E, \E_*\E)$ as sequences of coefficients in $p$-local integers,
\[
\Hom_\B(\E_{*-n}\E, \E_*\E) \cong \{ (a_m)_{m \in \N} \,\,|  a_m \in \mathbb{Z}_{(p)}\}=\mathbb{Z}_{(p)}^\mathbb{N}.
\]


For simplicity of notation, will denote a sequence $(a_m)_{m \in \N} $ by $\langle a_m \rangle$.    

\subsection{The formulas on sequences}\label{S:form}

To get the differential, we need to translate the following maps over to the sequence representation: 
\begin{eqnarray*}
\Psi^*: \mathbb{Z}_{(p)}^\mathbb{N} \longrightarrow\mathbb{Z}_{(p)}^\mathbb{N} \\
\Psi_*: \mathbb{Z}_{(p)}^\mathbb{N}\longrightarrow \mathbb{Z}_{(p)}^\mathbb{N}\\
\end{eqnarray*}

and furthermore,
\begin{eqnarray*}
\Psi^*:  \mathbb{Q} \longrightarrow\mathbb{Q} \\
q_*: \mathbb{Z}_{(p)}^\mathbb{N}  \longrightarrow \mathbb{Q} \\
q^*: \mathbb{Q} \longrightarrow \mathbb{Q}.
\end{eqnarray*}
 
\noindent {{\bf The map } $\bm{\Psi_*}$}:  We start by considering the map $\Psi_* = (\Psi^r-1)_*$ given by  composition with the map $\Psi^r-1$.   Chasing through our equivalences, we have 
\[ 
\xymatrix{
\Hom_{\E_*\E}(\E_{*}\E, \E_*\E)_t \ar[rr]^{(\Psi^r-1)_*} \ar[d]^{\simeq} && \Hom_{\E_*\E}(\E_{*}\E, \E_*\E)_t \ar[d]^{\simeq} \\
\Hom_{\E_*}(\E_{*}\E, \E_*)_t \ar[rr]^{(\Psi^r-1)_*} \ar[d]^{\simeq} && \Hom_{\E_*}(\E_{*}\E, \E_*)_t \ar[d]^{\simeq} \\
\Hom_{\mathbb{Z}_{(p)}}(\E_0\E, \mathbb{Z}_{(p)}[v_1^k])  \ar[rr] \ar[d]^{\simeq} && \Hom_{\mathbb{Z}_{(p)}}(\E_0\E, \mathbb{Z}_{(p)}[v_1^k]) \ar[d]^{\simeq} \\
\E^0 \E \ar[d]^{\simeq} \ar[rr] &&\E^0\E \ar[d]^{\simeq}\\
 \{ \sum\limits_{m\ge 0} a_m \Theta_m(\Psi^r) \,\,|\,\, a_m \in \mathbb{Z}_{(p)}\} \ar[rr] &&  \{ \sum\limits_{m\ge 0} a_m \Theta_m(\Psi^r) \,\,|\,\, a_m \in \mathbb{Z}_{(p)}\}
\\ } \]

 To calculate $(\Psi^r-1)_*$ we can work on the $v_1^k$ level.  Note that  when $k=0$, $\Psi^r$ acts as the identity, and so  $(\Psi^r-1)_* = 0$. For $k \neq 0$, we know that up to a $p$-local unit, \begin{align*} (\Psi^r-1)_*(v_1^k) & = (r^{k(p-1)} -1)v_1^k \\ & =  p^{\nu(k)+1}v^k_1  
\end{align*}
Therefore we can see that $(\Psi^r-1)_*$ is given by multiplication by $p^{\nu(k)+1}$.   

\bigskip
\noindent {{\bf The map } $\bm{\Psi^*}$} on $\Z_{(p)}$-sequences: Chasing through the effect of $(\Psi^r-1)^*$ is slightly more involved.  Starting with the $k=0$ case,  we see that since the vertical isomorphisms in the last step are ring isomorphisms, the overall effect on the sequences is multiplication by $\Theta_1(\Psi^r)$. In all that follows, we will write $\Theta_i$ in place of $\Theta_i(\Psi^r)$.  Then we can calculate:

\begin{align*}
\Theta_0\Theta_1 & = \Theta_1 \\
\Theta_m\Theta_1 & = (\Psi^r-1)(\Psi^r-r)(\Psi^r-r^{-1}) \cdots (\Psi^r-r^{\tilde{s}(m)})(\Psi^r-1) \\
\Theta_{m+1} & = (\Psi^r-1)(\Psi^r-r)(\Psi^r-r^{-1}) \cdots (\Psi^r-r^{\tilde{s}(m)})(\Psi^r-r^{\tilde{s}(m+1)}) \\
\end{align*}
where \[ \tilde{s}(m) =  \begin{cases} 
      \frac{m}{2} & m \textup{ even } \\
      \frac{1-m}2 & m \textup{ odd } \\
          \end{cases}
\]
So then 
\begin{align*}
\Theta_m\Theta_1  - \Theta_{m+1} & =  (\Psi^r-1)(\Psi^r-r)(\Psi^r-r^{-1}) \cdots (\Psi^r-r^{\tilde{s}(m)})(r^{\tilde{s}(m+1)}-1) \\
\Theta_m\Theta_1 & = [r^{\tilde{s}(m+1)}-1]\Theta_m + \Theta_{m+1}\\
\end{align*}
Therefore 
\begin{align*}
\sum_{m \geq 0} a_m\Theta_m\Theta_1 & =  a_0 \Theta_1 + \sum_{m \geq 1} a_m [r^{s(m)}-1]\Theta_m + \Theta_{m+1} &=\sum_{m \geq 1} (a_m(r^{s(m)}-1)+a_{m-1})\Theta_{m}\\
\end{align*}
where $s(m)=\tilde{s}(m+1)$.
Thus when $k=0$, our formula becomes 

$$
\Psi^{*}\langle a_{m} \rangle = \scaleleftright[1.75ex]{<} {\begin{array}{c}  0\\ a_{1}(r^{s(1)}-1 )+a_{0}\\ 
a_{2}(r^{s(2)}-1 )+a_{1}\\ 
\vdots \\
a_{m}(r^{s(m)}-1 )+a_{m-1} \\ \vdots \end{array} }{>}
$$


When $k\neq 0$, then we have $n \neq 0$ and thus, must determine how the map $\Psi^r-1$ behaves on $E(1)_{*-n}E(1)$ instead of just $E(1)_*E(1)$. Do to this, we observe what happens on the level of the generators $v_1^i$. We first note that

\[ (\Psi^r - 1)v_1^i = (r^{i(p-1)}-1)v_1^i.\]

As mentioned above, precomposing with such a map corresponds to multiplication by $\Theta_1$. Thus, upon shifting to $E(1)_{*-n}E(1)$ via multiplication by $v_1^k$, we see that multiplication by $\Theta_1$ would correspond to $\Psi^r-1$ producing $(r^{i(p-1)}-1)v_1^{i+k}$ in $E(1)_{*-n}E(1)$. However, to truly shift to working in $E(1)_{*-n}E(1)$ we observe that 

\[(\Psi^r - 1)v_1^{i+k} = (r^{(i+k)(p-1)}-1)v_1^{i+k}.\]

Due to this difference, precomposition with $\Psi^r-1$ on $E(1)_{*-n}E(1)$ when translated to sums of $\Theta_m$'s must include an additional $\Theta_0$ term. Up to a $p$-local unit, for any $i$, $$r^{(i+k)(p-1)} - r^{i(p-1)} = p^{\nu(k)+1}.$$ Hence, when $k\neq 0$, $\Psi^* $ acts by multiplication by $\Theta_1 + p^{\nu(k)+1}\Theta_0$. By performing a similar computation to the one above for $\sum_{m\geq 0} a_m\Theta_m\Theta_1$, we obtain

\[ \sum_{m\geq 0} a_m\Theta_m(\Theta_1 + p^{\nu(k)+1}\Theta_0) = p^{\nu(k)+1}a_0 + \sum_{m\geq 1}a_m[r^{{s}(m)}-1+p^{\nu(k)+1}]\Theta_m\]
 
and thus, when $k\neq 0$ our formula becomes

$$
\Psi^{*}\langle a_{m} \rangle = \scaleleftright[1.75ex]{<} { \begin{array}{c} p^{\nu(k)+1} a_0 \\ a_{1}(r^{s(1)}-1 + p^{\nu(k)+1})+a_{0}\\ 
\vdots \\
a_{m}(r^{s(m)}-1 +p^{\nu(k)+1})+a_{m-1} \\ \vdots \end{array}}{>}
$$


\bigskip
\noindent {{\bf The map } $\Psi^*$ on the rational terms}:

Let us consider $\Psi^*: \mathbb{Q} \longrightarrow \mathbb{Q}$, i.e. the map induced by $\Psi$ on terms of Type (II). We recall that we have an isomorphism
\[
\Hom_{\B}(\E_*\E, \E_* \otimes \mathbb{Q}) \cong \mathbb{Q}
\]
and that the $1 \in \mathbb{Q}$ on the right hand side corresponds to the map $q$ itself. But $q \circ \Psi =0$ as they are part of the resolution (\ref{inj-res}), thus the map $\Psi^*$ above is the zero map. 

\bigskip
For $t=s(2p-2)$, the copy of the rationals 
\[
\Hom_{\B}(\E_*\E, \E_{*+t} \otimes \mathbb{Q}) \cong \mathbb{Q}
\]
is generated by the function that sends $w \in \E_* \E$ to $0$ and $u \in \E_* \E$ to $v_1^{s}$. As
\[
(\Psi^r-1)v_1^{s}=(r^{(s)(p-1)}-1)v_1^{s},
\]
precomposition with $\Psi$ is multiplication by $r^{(s-1)(p-1)}-1$, which up to $p$-local unit is a nontrivial power of $p$. (This is also consistent with the case $t=0$, where this map is trivial.)

\bigskip
\noindent {{\bf The map } $\bm q_*$}: The map 
\[q_*: \Hom_\B(\E_{*-n}\E, \E_*\E) \longrightarrow \Hom_\B(\E_{*-n}\E, \E_*\otimes \mathbb{Q})\]
is the map obtained by composing with the map 
\[q:\E_*\E \longrightarrow \E_*\otimes\Q\] from the Adams-Baird-Ravenel resolution. The map $q$ is induced by the map $\E \rightarrow H\Q$ which, on homotopy, is a rational isomorphism in degree 0 and trivial in all other degrees.   So $q_*$ is induced by the inclusion $\mathbb{Z}_{(p)} \hookrightarrow \mathbb{Q}$ and becomes $q_*(\langle a_m \rangle ) =  a_0 $.  For $n \neq 0$, $q_*$ is trivial.

\bigskip
\noindent {{\bf The map } $\bm{q^*}$}:  Lastly we consider the map $$q^{*}: \Hom_\B(\E_*\otimes\Q, \E_*\otimes\Q) \longrightarrow \Hom_\B(\E_{*}\E, \E_*\otimes\Q).$$ We saw that both these terms are isomorphic to one copy of $\mathbb{Q}$ via the isomorphism $f \mapsto f(1)$. So, $q^*$ sends $1 \in \mathbb{Q}$ to the element in $\mathbb{Q}$ corresponding to the composite
\[
\E_*\E \xrightarrow{q} \E_*\otimes \mathbb{Q} \xrightarrow{1} \E_*\otimes \mathbb{Q}
\]
which is again $q$. Thus, $q^*: \mathbb{Q} \longrightarrow \mathbb{Q}$ is simply the identity map.

\section{The calcuation for $n = 0$} \label{sec:n0}

In this section and the next, we are going to use our explicit representations to calculate the homology of the endomorphism dga $C$.  Note that by our earlier remarks, we know that this should come out to   $H^*(C)=\pi_{-*}(L_1S^0)$  (the change in sign arises as the dga is cohomologically graded).

As explained at the end of Section \ref{sec:enddga}, in degrees around 0 our dga looks like
$$
\xymatrix{
 C^{-1} \ar[r]^{d^{-1}} & C^{0} \ar[r]^{d^{0}} & C^{1} \ar[r]^{d^{1}} & C^{2} \\
  & \Z^{\N}_{(p)} \ar[rd]^{\Psi_*} &  & \\
 \Z^{\N}_{(p)} \ar[ru]^{\Psi^*} \ar[rd]_{\Psi_*} &   & \Z^{\N}_{(p)} \ar[rd]^{q_*} & \\
  & \Z^{\N}_{(p)} \ar[ru]^{\Psi^*} \ar[rd]_{q_*} &   &{ \Q} \\
  &   & {\Q} \ar[ru]^{\Psi^*} & \\
  &\Q \ar[ru]_{q^{*}} && 
}
$$
Condensing it down we have 
$$ 
\xymatrix @R=.75pc{0 \ar[r] &\ds  \Z_{(p)}^{\N} \ar[r]^-{d^{-1}} & \ds \Z_{(p)}^{\N} \oplus \Z_{(p)}^{\N}\oplus \Q \ar[r]^-{d^{0}} & \ds\Z_{(p)}^{\N} \oplus \Q  
\ar[r]^-{d^{1}} & \ds\Q \ar[r]&0 \\
 & \framebox{-1} & \framebox{0} & \framebox{1} & \framebox{2} & }
$$

Now we need to see what each of these maps is on sequences.
Using our formulas   from Section \ref{S:form} we get 

\[ 
d^{-1}(\langle a_{m}\rangle) = (\Psi^{*}\langle a_{m}\rangle,\Psi_{*}\langle a_{m}\rangle ,0)  \\
 = \left(\scaleleftright[1.75ex]{<} { \begin{array}{c} 0 \\ a_{1}(r^{s(1)}-1)+a_{0}\\ 
\vdots \\
a_{m}(r^{s(m)}-1)+a_{m-1} \\ \vdots \end{array}}{>},  \scaleleftright[1.75ex]{<} { \begin{array}{c} 0 \\ 0\\ 
\vdots \\
0\\ \vdots
\end{array}}{>},  0 \right)
\] 

\begin{align*}
d^{0}(\langle a_{m} \rangle,\langle b_{m}\rangle ,x)
&  =(\Psi_{*}\langle a_{m}\rangle -\Psi^{*}\langle b_{m}\rangle ,q_{*}\langle b_{m}\rangle - q^{*}(x)) \\
 &  =\left( \scaleleftright[1.75ex]{<} { \begin{array}{cc}0  \\
-b_{1}(r^{s(1)}-1)-b_{0}\\
\vdots  \\
-b_{m}(r^{s(m)}-1)-b_{m-1} \\
\vdots \\
\end{array} }{>}, b_{0}-  x\right)
\end{align*}

\[ 
d^{1}(\langle a_{m} \rangle, y )
=q_{*}\langle a_{m}\rangle +\Psi^{*}(y) =a_{0}
\]

\begin{lemma}
The sequences of maps $d^{-1}$, $d^{0}$ and $d^{-1}$ give a cochain complex.
\end{lemma}
\begin{proof}
It is easy to see that $(d^{0}\circ d^{-1})(a_{m})=d^{0}(\Psi^{*}(a_{m}),0,0)=0$ and

\[ \begin{array}{ll}
(d^{1}\circ d^{0})(\langle a_{m}\rangle, \langle b_{m} \rangle ,x) & =d^1 \left( \scaleleftright[1.75ex]{<} { \begin{array}{c}0  \\
-b_{1}(r^{s(1)}-1)-b_{0}\\
-b_{2}(r^{s(2)}-1)-b_{1}\\
\vdots \\
-b_{m}(r^{s(m)}-1)-b_{m-1} \\
\vdots \\
\end{array} }{>} ,  
b_{0}-x \right)  \\ & = 0
\end{array}
\]

\end{proof}
\begin{theorem}\label{thm:n=0 homology}
Near $n=0$:  

\[H^{n}(C) = \begin{cases}0 & \mbox{ if } n = -1,  \\ \Z_{(p)} & \mbox{ if } n = 0  \\
0 & \mbox{ if } n = 1,  \\
\Q\slash\Z_{(p)} & \mbox{ if } n = 2  \\
							 \end{cases}\]
\end{theorem}

\begin{proof}
\noindent {$\bm{n=-1}$}:  Suppose that 
$\langle a_{m}\rangle\in \ker (d^{-1})$.  Then we know that $$a_{m}(r^{s(m)}-1)+a_{m-1} = 0 \,\,\,\mbox{for all} \,\,\, m \in \mathbb{N}.$$  We will show that $a_m = 0 $ for all $m \in \mathbb{N}$.   For any given $m$, choose $\ell>m$ such that $p(p-1) | s(\ell)$.  Then   $r^{s(\ell)}=1$, and so we know that $a_{\ell-1} = 0$.  Then since for all $j  \in \mathbb{N}$
$$a_{j}(r^{s(j)}-1)=-a_{j-1},$$ we see that if $a_{j} =0 $ then $a_{j-1}=0$ also.  So by induction, $a_m = 0$ also.  
So $d^{-1}$ is injective and $H^{-1}(C)=0$. 

\bigskip
\noindent {$\bm{n=0}$}:  Suppose that $(\langle a_m \rangle, \langle b_m \rangle, x) \in \ker d^0$.  Then we know that \[ \left( \scaleleftright[1.75ex]{<} { \begin{array}{cc}0  \\
-b_{1}(r^{s(1)}-1)-b_{0}\\
\vdots  \\
-b_{m}(r^{s(m)}-1)-b_{m-1} \\
\vdots & \vdots \\
\end{array} }{>},  b_0 -x \right) = 0\] 
Therefore $b_{0}=x$ and $b_{m}=0,\forall m \in \N$ by the same argument used for $n=-1$. Then $(\langle a_m \rangle, \langle b_{m} \rangle, x)=(\langle a_m \rangle, 0)$  

We claim that \begin{equation}\label{argument}
\begin{array}{lll}
\Ima d^{-1}& =&\{ d^{-1}\langle c_{m} \rangle: \langle c_{m} \rangle \in  \Z_{(p)}^{\mathbb{N}} \}\\
& =&\{(\Psi^{*}\langle c_{m} \rangle, \langle 0 \rangle,0 ): \langle c_{m} \rangle \in  \Z_{(p)}^{\mathbb{N}} \}\\
& =&\{(\langle a_m \rangle, \langle 0 \rangle, 0): \langle a_m \rangle \in  \Z_{(p)}^{\mathbb{N}}, \ a_{0}=0 \}.
\end{array}
\end{equation}
It is clear that any element in the image must have $a_0 = 0$.  Conversely, given  $\langle a_m \rangle$ with $a_0 = 0$, we can produce $\langle c_m \rangle$ such that $d^{-1} (\langle c_m \rangle )  = (\langle a_m \rangle, \langle 0 \rangle, 0)$.   We produce $c_m$ as follows:   for any $m$, we choose the smallest $\ell > m $ such that $(p-1)p | \ell$.  Then we need  to choose $c_{\ell-1 }$ satisfying $c_{\ell-1 } = a_{\ell}$.    Then we work our way down, observing that if we have chosen $c_j$, we can then find $c_{j-1}$ to satisfy $$c_j(r^{s(j)}-1) + c_{j-1} = a_{j-1}.$$  Inductively we can get a value for $c_m$.

So we can find $\langle c_m \rangle$ such that $c_m = a_{m}(r^{s(m)}-1)+a_{m-1}$ and hence $d^{-1}(\langle c_m \rangle) = (\langle a_m \rangle, \langle 0 \rangle, 0)$.
Thus we see that $\ker d^0 / \Ima d^{-1} = \Z_{(p)}$ as represented by the value of $a_0$ in $(\langle a_m \rangle, \langle 0 \rangle, 0)$.

\bigskip
\noindent {$\bm{n=1}$}: Note that $\ker d^1 = (\langle a_m\rangle, y) $ such that  $a_0  = 0$. 
We can see in our claim above, given in equations \eqref{argument},  that there exists $\langle b_m\rangle =\langle -c_m\rangle \in \Z_{(p)}^{\mathbb{N}}$ such that $$\langle a_m\rangle=\Psi^{*}(\langle c_m\rangle)=-\Psi^{*}(\langle b_m\rangle).$$ 
Then  $$d^{0}( \langle 0 \rangle, \langle b_{m} \rangle,b_{0}-y) = (\langle a_m \rangle, y).$$ As  $\ker d^1 = \Ima d^{0}$ we get $H^{1}(C)=0.$  

\bigskip
\noindent {$\bm{n=2}$}:  Finally, we know that $\ker d^2 =  \Q$.  Clearly $\Ima d^1=\Z_{(p)}$ . So  $$H^2 (C) = \Q\slash\Z_{(p)}.$$  
\end{proof}

\section{The Homology Calculation for $n\neq 0$} \label{sec:nnot0}

Looking back on  our description of the endomorphism dga at the end of Section \ref{sec:enddga}, we see that in terms of our sequence representations, we have 
$$
\xymatrix{
 C^{(2p-2)k-1} \ar[r] & C^{(2p-2)k} \ar[r] &C^{(2p-2)k+1} \ar[r] & C^{(2p-2)k+2} \\
&   \Z^{\N}_{(p)} \ar[rd]^{\Psi_*} & & \\
 \Z^{\N}_{(p)} \ar[ru]^{\Psi^*} \ar[rd]_{\Psi_*} &  & \Z^{\N}_{(p)} \ar[rd]^{q_* = 0} & \\
 &  \Z^{\N}_{(p)} \ar[ru]^{\Psi^*} \ar[rd]_{q_* = 0} &  & \mathbb{Q}\\
  &   & \mathbb{Q} \ar[ru]^{\Psi^*}_{\cong} & \\
 &  0\ar[ru]_{q^{*}} && 
}
$$

Condensing down our earlier diagram, we are looking at 
\begin{equation}\label{E:kn0} \xymatrix @R=.75pc{0 \ar[r] &\ds  \Z_{(p)}^{\N} \ar[r]^-{d^{(2p-2)k-1}} & \ds\Z_{(p)}^{\N} \oplus \Z_{(p)}^{\N} \ar[r]^-{d^{(2p-2)k}} & \ds\Z_{(p)}^{\N} \oplus\mathbb{Q}  \ar[r]^-{d^{(2p-2)k+1}} & \ds  \mathbb{Q} \ar[r] & \ds 0 \\
 & \framebox{(2p-2)k-1} & \framebox{(2p-2)k} & \framebox{(2p-2)k+1} & \framebox{(2p-2)k+2} & }\end{equation}

where 
\[ \begin{array}{ll}
d^{(2p-2)k-1}(\langle a_m \rangle)  & = (\Psi^*\langle a_m \rangle,\Psi_*\langle a_m \rangle)  \\
 & = \left(\scaleleftright[1.75ex]{<} { \begin{array}{c} p^{\nu(k)+1}a_0\\ a_{1}(r^{s(1)}-1+p^{\nu(k)+1})+a_{0}\\ 
\vdots \\
a_{m}(r^{s(m)}-1+p^{\nu(k)+1}) + a_{m-1} \\ \vdots \end{array}}{>},  \scaleleftright[1.75ex]{<} { \begin{array}{c} p^{\nu(k)+1}a_0\\ p^{\nu(k)+1}a_1\\\vdots \\ p^{\nu(k)+1}a_m\\ 
\vdots \\
\end{array}}{>} \right) \end{array}
\]

\[ \begin{array}{ccc} d^{(2p-2)k}(\langle a_m \rangle,\langle b_m \rangle) & =& (\Psi_*\langle a_m \rangle - \Psi^*\langle b_m \rangle, 0)
\\
\\
&  =& \left(\scaleleftright[1.75ex]{<} { \begin{array}{cc}p^{\nu(k)+1}a_0 - p^{\nu(k)+1}b_0 \\
p^{\nu(k)+1}a_1-b_{1}(r^{s(1)}-1+p^{\nu(k)+1})-b_{0}\\
\vdots  \\
p^{\nu(k)+1}a_m-b_{m}(r^{s(m)}-1+p^{\nu(k)+1})-b_{m-1} \\
\vdots \\
\end{array} }{>}, 0 \right)
\end{array}\]
and 
\[ 
d^{(2p-2)k+1}( \langle a_m \rangle, b)= p^{\nu(k)+1}b).
\]

We start by verifying the following:  
\begin{lemma}\label{lem:kneq0 complex}
The sequence of modules and maps described in (\ref{E:kn0})  is a cochain complex.
\end{lemma}

\begin{proof}
 We show that  $d^{(2p-2)k}( d^{(2p-2)k-1}\langle a_m \rangle) = 0$ for any sequence $\langle a_m \rangle$, where $a_m\in \Z_{(p)}$:   
$$
 \begin{array}{lllll}
\Psi_*(\Psi^*\langle a_m \rangle) &=&\Psi_*\left(\scaleleftright[1.75ex]{<} { \begin{array}{c} p^{\nu(k)+1}a_0\\ a_{1}(r^{s(1)}-1+p^{\nu(k)+1})+a_{0}\\ 
\vdots \\ a_{m}(r^{s(m)}-1+p^{\nu(k)+1})+ a_{m-1} \\ \vdots \end{array}}{>}\right)\\&=&
\scaleleftright[1.75ex]{<} { \begin{array}{c} p^{2\nu(k)+1}a_0\\ p^{\nu(k)+1}a_{1}(r^{s(1)}-1+p^{\nu(k)+1})+p^{\nu(k)+1}a_{0}\\ 
\vdots \\
p^{\nu(k)+1}a_{m}(r^{s(m)}-1+p^{\nu(k)+1}) + p^{\nu(k)+1}a_{m-1} \\ \vdots \end{array}}{>}\\
\\
&=& \Psi^{*}\left( \scaleleftright[1.75ex]{<} { \begin{array}{c} p^{\nu(k)+1}a_0\\ p^{\nu(k)+1}a_1\\ \vdots \\ p^{\nu(k)+1}a_m\\ 
\vdots \\
\end{array}}{>} \right) =\Psi^*(\Psi_*\langle a_m \rangle).
\end{array}
$$
Then $$ d^{(2p-2)k}( d^{(2p-2)k-1}\langle a_m \rangle)=\Psi_*(\Psi^*\langle a_m \rangle) -\Psi^*(\Psi_*\langle a_m \rangle)=0.$$

\bigskip
Also, 
\[
d^{(2p-2)k+1} \circ d^{(2p-2)k}(\langle a_m \rangle, \langle b_m \rangle)=d^{(2p-2)k+1}( p^{\nu(k)+1}a_0 - p^{\nu(k)+1}b_0, 0)= 0
\]
as required.


\end{proof}

It had been immediately clear that the isomorphism $\Psi^*=p^{\nu(k)+1}: \mathbb{Q} \longrightarrow \mathbb{Q}$ does not contribute anything to the chain complex, so we will omit it from here onwards.

\bigskip
Before verifying that the cohomology is as expected, we examine the kernel of $ d^{(2p-2)k}$ more closely.  
\begin{lemma}\label{lem:p divides bn}
For all $(\langle a_m \rangle ,\langle b_m\rangle )\in \ker d^{(2p-2)k}$, $p^{\nu(k)+1}|b_m$ for all $m \in \N$. 
\end{lemma}

\begin{proof}
If $(\langle a_m \rangle,\langle b_m \rangle)$ is in the kernel, we know that 
\[p^{\nu(k)+1}a_0 = p^{\nu(k)+1} b_0 \] and \[ p^{\nu(k)+1}a_m = (r^{s(m)}-1+p^{\nu(k)+1})b_m + b_{m-1}\mbox { \, \, for all\, \, } m \geq 1.\] 
 Since $r \in (\Z/p^2)^{\times}$, we know $r^{s(m)} - 1 = 0$ whenever $s(m)$ is a multiple of $p(p-1)$. Now fix $m \in \N$ and we will show that $p^{\nu(k)+1}|b_m$. Let $\ell \in \N$, $\ell>m$ such that $r^{s(\ell)} - 1 = 0$. Then $$p^{\nu(k)+1}a_{\ell} = p^{\nu(k)+1}b_{\ell} + b_{\ell-1}$$ and thus, $p^{\nu(k)+1}|b_{\ell-1}$. Then since 
\[p^{\nu(k)+1}a_q= (r^{s(q)}-1+p^{\nu(k)+1})b_{q} + b_{q-1}\] it is clear that if $p^{\nu(k)+1}|b_q$ then also $p^{\nu(k)+1}|b_{q-1}$ for any $q \geq 1$. Thus since $p^{\nu(k)+1}| b_{\ell} $ and $\ell > m$,  $p^{\nu(k)+1}|b_m$ by induction.
\end{proof}

\begin{theorem} \label{thm:n neq 0 homology}
When $k \neq 0$, 

\[H^{n}(C) = \begin{cases} \Z/p^{\nu(k)+1} & \mbox{ if } n = (2p-2)k+1 \\
							0 & \mbox{ else} \end{cases}\]
\end{theorem}

\begin{proof}From the complex, it is immediate that $H^t(C) = 0$ for all $t$ that are not congruent to $-1$, $0$ or $1$ modulo $2p-2$. 
 \bigskip
 
\noindent {$\bm{n=(2p-2)k-1}$}: Suppose $\langle a_m \rangle$ is in $\ker d^{(2p-2)k-1}$.  
Then $p^{\nu(k)+1}a_m = 0$ for all $m\geq 0$, and so $a_m = 0$ for all $m\geq 0$. Thus $\ker d^{(2p-2)k-1}=0$ and so $H^{(2p-2)k-1}(C) = 0$.

\bigskip
\noindent {$\bm{n=(2p-2)k}$}:  Let $(\langle a_m \rangle ,\langle b_m\rangle ) \in \ker d^{(2p-2)k}$.   This means  \[ p^{\nu(k)+1}a_0 = p^{\nu(k)+1} b_0, \mbox{ so } a_0 = b_0, \] and \[ p^{\nu(k)+1}a_m = (r^{s(m)}-1+p^{\nu(k)+1})b_m + b_{m-1}  \mbox{ for all } m \geq 1.\]  By Lemma \ref{lem:p divides bn} we know $p^{\nu(k)+1}|b_m$ for all $m$,  and so we may write $$b_m = p^{\nu(k)+1}c_m \,\,\,\mbox{for some}\,\,\, c_m \in \Z_{(p)}.$$ Thus, $a_0 = b_0 = p^{\nu(k)+1}c_0$,  and for $m \geq 1$ we may write \[ a_m = (r^{s(m)}-1+p^{\nu(k)+1})c_m + c_{m-1}. \] Thus we see that all elements of the kernel are of the form

\[ d^{(2p-2)k-1}(\langle c_m\rangle  ) = \left(\scaleleftright[1.75ex]{<} { \begin{array}{c}p^{\nu(k)+1}c_0 \\ (r^{s(m)}-1+p^{\nu(k)+1})c_1 + c_{0}\\ 
\vdots \\
(r^{s(m)}-1+p^{\nu(k)+1})c_m + c_{m-1}\\ \vdots \end{array}}{>},  \scaleleftright[1.75ex]{<} { \begin{array}{c} p^{\nu(k)+1}c_0\\  p^{\nu(k)+1}c_1\\\vdots \\ p^{\nu(k)+1}c_m\\ 
\vdots \\
\end{array}}{>} \right)\] Thus $\ker d^{(2p-2)k} = \Ima  d^{(2p-2)k-1} $ and $H^{(2p-2)k}(C) = 0$.

\bigskip

\noindent {$\bm{n=(2p-2)k + 1}$}: 
For any $(\langle a_m \rangle ,\langle b_m \rangle  \in \Z^{\N}_{(p)} \oplus \Z^{\N}_{(p)}$ we have
\[ d^{(2p-2)k}(\langle a_m\rangle ,\langle b_m \rangle ) =  \scaleleftright[1.75ex]{<} {\begin{array}{c} p^{\nu(k)+1}(a_0-b_0) \\
p^{\nu(k)+1}a_1-(r^{s(1)}-1+p^{\nu(k)+1})b_1-b_{0}\\
\vdots \\
p^{\nu(k)+1}a_m-(r^{s(m)}-1+p^{\nu(k)+1})b_m-b_{m-1}\\
\vdots \\
\end{array}} {>} .\]
So if $\langle c_m\rangle  \in \Z^{\N}_{(p)}$ is in $\Ima d^{(2p-2)k}$, then $c_0$ is clearly divisible by $p^{\nu(k)+1}$.   We will show that the converse is also true:  if  
 $p^{\nu(k)+1}|c_0$, then there exist sequences $\langle a_m
\rangle,\langle b_m \rangle \in \Z^{\N}_{(p)}$ such that $$d^{(2p-2)k}(\langle a_m \rangle, \langle b_m \rangle) = \langle c_m\rangle. $$

Given any $b_0$, we may always select $a_0$ so that $p^{\nu(k)+1}(a_0-b_0)=c_0$. We will show that we can find $a_m, b_m,$ and $b_{m-1}$ so that 
\[p^{\nu(k)+1}a_m-(r^{s(m)}-1+p^{\nu(k)+1})b_m-b_{m-1} = c_{m}\] compatibly for all $m \geq 1$. As in the proof of Lemma \ref{lem:p divides bn}, for any fixed $m \in \N$, we may choose the smallest  value $\ell > m$ such that $r^{s(\ell)} - 1 = 0$. Then if we take $a_\ell = b_\ell$ and $b_{\ell-1} = -c_\ell$ we have 
\[p^{\nu(k)+1}a_{\ell}-(r^{s(\ell)}-1+p^{\nu(k)+1})b_\ell-b_{\ell-1} = c_{\ell}. \]
Now suppose we have defined $a_q, b_q,$ and $b_{q-1}$ so that 
\[p^{\nu(k)+1}a_q-(r^{s(q)}-1+p^{\nu(k)+1})b_q-b_{q-1} = c_{q}. \]
If we then let $a_{q-1} = b_{q-1}$ and $b_{q-2} = (r^{s(q-1)}-1)b_{q-1} - c_{q-1}$ we will obtain
\[p^{\nu(k)+1}a_{q-1}-(r^{s(q-1)}-1+p^{\nu(k)+1})b_{q-1}-b_{q-2} = c_{q-1}.\]
Again, inducting downwards from $\ell$ shows that we can find values for $a_m, b_m$ for any $m$ such that 
 $d^{(2p-2)k}(\langle a_m \rangle, \langle b_m \rangle) = \langle c_m\rangle$.  
\end{proof}

\section{Products and Massey Products} \label{sec:massey}

In this section we discuss the multiplicative structure of $C$, showing that it induces an injective multiplication $H^{-(2p-2)k+1}(C)\otimes H^{(2p-2)k+1}(C) \rightarrow H^{2}(C)$ and that $C$ has the appropriate Massey products.

\subsection{Products} \label{sec:products}
  In this section we will prove the following:

\begin{proposition}\label{prop:mult}
The multiplication $C^{-(2p-2)k+1}\otimes C^{(2p-2)k+1}\rightarrow C^2$ induces multiplication $H^{-(2p-2)k+1}(C)\otimes H^{(2p-2)k+1}(C)\rightarrow H^2(C)$  given by

\[\xymatrix@R=.5pc{\Z/p^{\nu(k)+1} \otimes \Z/p^{\nu(k)+1} \ar[r] & \Q/\Z_{(p)}\\
a\otimes b \ar@{|->}[r] & \frac{a}{p^{\nu(k)+1}}\frac{b}{p^{\nu(k)+1}}}\]

\end{proposition}

This will immediately give the following:

\begin{corollary}
The multiplication $H^{-(2p-2)k+1}(C)\otimes H^{(2p-2)k+1}(C)\rightarrow H^2(C)$ is injective.
\end{corollary}

In order to prove Proposition \ref{prop:mult}, 
  we examine the multiplication on  $C^*$. The multiplication $C^{-(2p-2)k+1}\otimes C^{(2p-2)k+1}\rightarrow C^2$ is of the form
\[\xymatrix{\Hom_\B(\E_{*-n}\E, \E_*\E) \otimes \Hom_\B(\E_{*+n}\E,  \E_*\E) \ar[d] \\ \Hom_\B(\E_*\E, \E_* \otimes \Q)}\]
given by the composition of morphisms in $\B$. 


To obtain the product in $\Hom_\B(\E_*\E, \E_* \otimes \Q)$, we compose   with $q$.

We translate this into a  product on our sequence representations.  

\begin{lemma}\label{l:mult}
For sequences $\langle a_m \rangle$ and $\langle b_m \rangle$ representing  $\sum_{m\geq 0} a_m \Theta_m$ and $\sum_{n\geq 0} b_n \Theta_n$ in $\E^t\E$ and $\E^s\E$ respectively, where $t = (2p-2)k$ and $s = (2p-2)\ell$, 

\[ \sum_{m\geq 0} a_m \Theta_m \cdot \sum_{n\geq 0} b_n \Theta_n = \sum_{m+n=i} a_m b_n \Theta_m \Theta_n p^{N(i+k,m)-N(i,m)+N(i+\ell,n)-N(i,n)}\]
where $N(i, k)$ are integers that depend on $i$ and $k$.  
\end{lemma}

\begin{proof}
Recall that when  $n = (2p-2)k$ we obtain the sequences using the  equivalence 
\[\Hom_{\B}(\E_{*-n}\E, \E_*\E) \cong \E^n\E = \E^0\E \cdot v_1^k\]
So we consider
\[ 
\xymatrix{
\E^t \E \otimes \E^{s} \E \ar@{=}[d] \ar[rr] && \E^{s+t}\E \ar@{=}[d]\\
\E^0\E \cdot v_1^k \otimes \E^0\E \cdot v_1^{\ell} \ar[rr] && \E^0\E \cdot v_1^{k+\ell} } \]
for the product of elements from $C^{t}$ and $C^{s}$ where $t = (2p-2)k$ and $s = (2p-2)\ell$.  
Since $\E^0\E = \{ \sum_{m \geq 0} a_m \Theta_m\}$ where $\Theta_m = \Theta_m(\Psi^r - 1)$,  we need to understand how $\sum_{m\geq 0} a_m \Theta_m$ acts on $v_1^i$. 

If $t = 0$ then 
\begin{align*} \sum_{m\geq 0} a_m \Theta_m \cdot v_1^i&  = \sum_{m\geq 0} a_m (r^{i(p-1)} - 1)(r^{i(p-1)} - r) \cdots (r^{i(p-1)} - r^{s(m)})v_1^i \\ & = \sum_{m \geq 0} a_m p^{N(i,m)} v_1^i \end{align*}
where $N(i,m)$ is some integer depending on $i$ and $m$.  If $t = (2p-2)k$  for $k \neq 0$,
\begin{align*} \sum_{m\geq 0} a_m \Theta_m \cdot v_1^i &= (\sum_{m\geq 0} a_m \Theta_m \cdot v_1^{i+k})v_1^{-k} \\&= \sum_{m\geq 0} a_m (r^{i(p-1)} - 1)(r^{i(p-1)} - r) \cdots (r^{i(p-1)} - r^{s(m)})v_1^i \\
&= \sum_{m \geq 0} a_m p^{N(i+k,m)} v_1^i\end{align*}
Applying this to the sum yields the product described in the lemma.  

\end{proof}

\begin{corollary} \label{cor:mult2}
 The degree term in the sequence $\langle a_m \rangle \cdot \langle b_n \rangle$ is $a_0 b_0$.
\end{corollary}
\begin{proof} From the definition we see that  $N(i,0) = 0$ for any $i$, since $v_1^0 = 1$.   Since the only way for  $\Theta_m \Theta_n = \Theta_0$ is to have  $m = n = 0$, this proves the claim.    
\end{proof}

\begin{proof}[Proof of Propositon \ref{prop:mult}]
We saw in the homology computation of Theorems \ref{thm:n=0 homology} and \ref{thm:n neq 0 homology} that the homology in  $H^{(2p-2)k + 1}(C) $ and $H^2(C)$ is represented by the value of the index zero term in the sequences.  Thus, to compute a product $$H^{-(2p-2)k+1}(C)\otimes H^{(2p-2)k+1}(C)\rightarrow H^2(C)$$ we need only consider the multiplication $$C^{-(2p-2)k+1}\otimes C^{(2p-2)k+1}\rightarrow C^2$$ on the index zero terms of sequences.    By Corollary \ref{cor:mult2}, if $\langle a_n \rangle \cdot \langle b_n \rangle = \langle c_n \rangle$ then $c_0 = a_0 b_0$. Therefore if we pick any $$a \in H^{-(2p-2)k+1}(C) = \Z/p^{\nu(k)+1} \,\,\,\mbox{and}\,\,\,b \in H^{(2p-2)k+1}(C) = \Z/p^{\nu(k)+1},$$ we know multiplying them will yield the product in the quotient in $H^2(C) = \Q/\Z_{(p)}$.  Explicitly,  we first consider $a$ and $b$ as  $\frac{a}{p^{\nu(k)+1}} \in \Q/\Z_{(p)}$ and $\frac{b}{p^{\nu(k)+1}}\in \Q/\Z_{(p)}$ respectively, in $\Q/\Z_{(p)}$ and then multiply these representatives together.
\end{proof}

\subsection{Massey Products}  Here we calculate the Massey products.

\begin{proposition} Suppose that 
 $\gamma_k$ denotes an element of the cohomology $$H^{(2p-2)k+1}(C)\cong \Z/p^{\nu(k)+1}$$ such that $p \gamma_k = 0$.  Then 
the $\gamma_k$'s satisfy the following Massey product relation:
\[ \langle \gamma_i, p, \gamma_j\rangle = \gamma_{i+j}\]
and the indeterminancy of this product is zero.
\end{proposition}

\begin{proof}

We will compute the product directly using the definition of Massey product.
The cohomology class $\gamma_k$ must be a multiple of $p^{\nu(k)}$,  and we can represent it by the cycle 
\[a=   \scaleleftright[1.75ex]{<} { \begin{array}{c} p^{\nu(k)+1}a_0\\ 0\\ 0 \\
\vdots \\
\end{array}}{>} \]
where $a_0$ is some value such that $\nu(a_0)=0$.   Choosing the analogous representative  for $\gamma_j$ gives 
 the following cycles, $a$, $b$, and $c$, representing $\gamma_i$, $p$, and $\gamma_j$ respectively:
\[ a =   \scaleleftright[1.75ex]{<} { \begin{array}{c} p^{\nu(k)+1}a_0\\ 0\\ 0 \\ 
\vdots \\
\end{array}}{>}  
\hspace{1cm}
b=p
\hspace{1cm}
c =    \scaleleftright[1.75ex]{<} { \begin{array}{c} p^{\nu(k)+1}c_0\\  0 \\ 0\\ 
\vdots \\
\end{array}}{>}    \]
where $\nu(a_0) = \nu(c_0)=0$. 

\bigskip Now we choose
\[ u = \scaleleftright[1.75ex]{<} { \begin{array}{c} a_0\\ 0\\ 0 \\  
\vdots \\
\end{array}}{>}  
\hspace{1cm}
\mbox{and}
\hspace{1cm}
v =   \scaleleftright[1.75ex]{<} { \begin{array}{c} -c_0\\ 0\\ 0 \\ 
\vdots \\
\end{array}}{>} \]
where $|u| = (2p-2)i$ and $|v| = (2p-2)j$.
We can compute 
\[
d^{(2p-2)i}(u) = \scaleleftright[1.75ex]{<} { \begin{array}{c} p^{\nu(i)+1} a_0\\ 0\\ 0 \\ 
\vdots \\
\end{array}}{>}  = pa = (-1)^{1+|a|}a\cdot b
\]
and
\[
d^{(2p-2)j}(v) =   \scaleleftright[1.75ex]{<} { \begin{array}{c} -p^{\nu(j)+1}c_0 \\ 0\\ 0 \\ 
\vdots \\
\end{array}}{>} = -pc = (-1)^{1+|b|}b\cdot c.
\]

Therefore the Massey product $ \langle \gamma_i, p, \gamma_j\rangle $ can be computed as $[(-1)^{1+|u|}u\cdot c + (-1)^{1+|a|}a\cdot v]$.  This gives us 
\[-  \scaleleftright[1.75ex]{<} { \begin{array}{c} a_0\\ 0\\ 0 \\  
\vdots \\
\end{array}}{>} \cdot  \scaleleftright[1.75ex]{<} { \begin{array}{c} p^{\nu(k)+1}c_0\\  0 \\ 0\\ 
\vdots \\
\end{array}}{>}  +   \scaleleftright[1.75ex]{<} { \begin{array}{c} p^{\nu(k)+1}a_0\\ 0\\ 0 \\ 
\vdots \\
\end{array}}{>}  \cdot   \scaleleftright[1.75ex]{<} { \begin{array}{c} -c_0\\ 0\\ 0 \\ 
\vdots \\
\end{array}}{>} \]
which yields
\[   \scaleleftright[1.75ex]{<} { \begin{array}{c} -2a_0c_0(p^{\nu(i)}+p^{\nu(j)})\\ 0\\ 0 \\ 
\vdots \\
\end{array}}{>} 
\]
by our description of the multiplication in Section \ref{sec:products}. 

Now we can rewrite $p^{\nu(i)}+p^{\nu(j)}$ as 
\[p^{\nu(i)}+p^{\nu(j)} = p^{\mbox{min}(\nu(i),\nu(j))}(1 + p^{\mbox{max}(\nu(i),\nu(j))-\mbox{min}(\nu(i),\nu(j))}).\]
If $i \neq j$ then $\nu(i+j) = \mbox{min}(\nu(i),\nu(j))$ so $p^{\nu(i) + \nu(j)} = p^{\nu(i+j)}m$ where $\nu(m) = 0$. If $i = j$ then $\nu(i+j) = \nu(2i) = \nu(2) + \nu(i) = \nu(i)$. Thus, in this case $p^{\nu(i) + \nu(j)} = 2p^{\nu(i)} = 2p^{\nu(i+j)}$.

Thus, 
\[(-1)^{1+|u|}u\cdot c + (-1)^{1+|a|}a\cdot v=    \scaleleftright[1.75ex]{<} { \begin{array}{c} -2a_0c_0m(p^{\nu(i+j)})\\ 0\\ 0 \\ 
\vdots \\
\end{array}}{>} 
\]
where $m$ is some value such that $\nu(m)=0$. Thus we also have $\nu(2a_0c_0m)=0$ so this is an element of $H^{(2p-2)(i+j) + 1}(C)$ of order $p$ which represents $\gamma_{i+j}$.

Finally, we note that the indeterminancy of the product is
\[ \gamma_i H^{(2p-2)j}(C) \oplus \gamma_j H^{(2p-2)i}(C) \]
which is zero because the cohomology in each of those degrees is zero.
\end{proof}

\end{document}